\documentclass[11pt,amsfonts]{article}
\usepackage{natbib}
\usepackage{graphicx}
\usepackage{latexsym}
\usepackage{amssymb}
\usepackage{amsmath}
\usepackage{enumerate}
\usepackage{color}
\usepackage{layout}
\usepackage{eufrak}
\usepackage{float}
\usepackage{dsfont}
\usepackage{xspace}

\newcommand{\ddelta}{\ensuremath{\Delta\!\!\!\Delta}\xspace}

\newtheorem{prop}{Proposition}

\newtheorem{theorem}{Theorem}
\newtheorem{remark}{Remark}

\def\real{{\mathord{{\rm I\kern-2.8pt R}}}}        
\def\inte{{\mathord{{\rm I\kern-2.8pt N}}}}

\def\sZZ{{\rm Z\kern-2.8ptem{}Z}}

\def\z{{\mathchoice
  {\sZZ}
  {\sZZ}
  {\rm Z\kern-0.30em{}Z}
  {\rm Z\kern-0.25em{}Z} }}
\def\sQQ{{\kern 0.27em \vrule height1.45ex width0.03em depth0em
          \kern-0.30em \rm Q}}
\def\qu{{\mathchoice
    {\sQQ}
    {\sQQ}
  {\kern 0.225em \vrule height1.05ex width0.025em depth0em \kern-0.25em \rm Q}
  {\kern 0.180em \vrule height0.78ex width0.020em depth0em \kern-0.20em \rm Q}
        }}
\def\sCC{{\kern 0.27em \vrule height1.45ex width0.03em depth0em
          \kern-0.30em \rm C}}
\def\complex{{\mathchoice
    {\sCC}
    {\sCC}
  {\kern 0.225em \vrule height1.05ex width0.025em depth0em \kern-0.25em \rm C}
  {\kern 0.180em \vrule height0.78ex width0.020em depth0em \kern-0.20em \rm C}
        }}

\newcommand{\ba}{\begin{array}}
\newcommand{\ea}{\end{array}}
\newcommand{\be}{\begin{equation}}
\newcommand{\ee}{\end{equation}}
\newcommand{\bea}{\begin{eqnarray}}
\newcommand{\eea}{\end{eqnarray}}
\newcommand{\beaa}{\begin{eqnarray*}}
\newcommand{\eeaa}{\end{eqnarray*}}

\def\z{\zeta}

\font\tenmath=msbm10 \font\sevenmath=msbm7 \font\fivemath=msbm5
\newfam\mathfam \textfont\mathfam=\tenmath
\scriptfont\mathfam=\sevenmath \scriptscriptfont\mathfam=\fivemath

\def \={{\buildrel {\rm (law)} \over =}}

\newcommand{\basa}{\begin{assumption}}
\newcommand{\easa}{\end{assumption}}

\newcommand{\bas}{\begin{assum}}
\newcommand{\eas}{\end{assum}}

\def\Var{\operatorname{Var}}

\newcommand{\ignore}[1]{}
\begin{document}

\renewcommand{\thefootnote}{\fnsymbol{footnote}}

\renewcommand{\thefootnote}{\fnsymbol{footnote}}

\title{Quadratic variations for Gaussian isotropic random fields on the sphere}
\author{Radomyra Shevchenko \vspace*{0.2in} \\
 Fakult\"at f\"ur Mathematik,
LSIV, TU Dortmund \\
 Vogelpothsweg 87, 44227 Dortmund, Germany \\
\quad radomyra.shevchenko@tu-dortmund.de \vspace*{0.1in}\\
\vspace*{0.1in} 
}

\maketitle

\begin{abstract} 
In this paper we define (empirical) quadratic variations for a Gaussian isotropic random field $f$ on a unit sphere as sums over equidistant increments on one single geodesic line on the surface of the sphere. We prove a noncentral limit theorem for a fixed Fourier component of such a field as well as quantitative central limit theorems in the increasing frequency regime. Based on these results we propose estimators of the angular power spectrum and study their properties. Moreover, we show a quantitative central limit theorem for quadratic variations over the field $f$ and construct an estimator for the Hurst parameter of a $L^2(\mathbb S^2)$-valued fractional Brownian motion.
\end{abstract}





\section{Introduction}
The behaviour of isotropic spherical random fields is a subject that has attracted a lot of attention in recent years, in particular since such fields can be used to model and describe the phenomenon of cosmic microwave background (CMB). One of the most widely considered setups is that of a Gaussian isotropic random field $\{f(x),\, x\in\mathbb S^2\}$ on the unit sphere. It is well-known that for such a field $f$ a spectral decomposition with respect to eigenfunctions of the spherical Laplacian holds almost surely and in $L^2$, namely
$$f(x)=\sum_{\ell=1}^\infty f_\ell(x)=\sum_{\ell=1}^\infty \sum_{m=-\ell}^\ell a_{m\ell}Y_{m\ell}(x), $$
where $Y_{m\ell}$ are the spherical harmonics and the random coefficients $a_{m\ell}$ are complex Gaussian and such that $\mathbb E [a_{m\ell}\overline{a_{m'\ell'}}]=C_\ell\delta_{\ell\ell'}\delta_{mm'}$ (see e.g. \cite{MP}). The sequence $C_\ell$ is known as the angular power spectrum of $f$.

A plethora of recent works in this field is concerned with limit theorems for functionals of Fourier components or related objects with the asymptotics of $\ell$ tending to infinity. Examples include Minkowski functionals, i.e. the area of excursion sets (\cite{MW, T}), the nodal length (\cite{W09, W, MRW, MRT}) and the Euler-Poincar\'e characteristic (\cite{CMW}), and the so-called polyspectra, i.e. spatial integrals over the powers of the different components or their weighted averages known as needlets (\cite{MRW, CM}). The interest in such statistics is fueled by Berry's conjecture relating the high-frequency Fourier components to random waves (known as Berry's random wave model), see the survey \cite{W11} for further details. An important question in view of statistical applications for the CMB is the estimation of the angular power spectrum and its asymptotics for growing frequencies. There are several works dealing with this question (e.g. \cite{FG, LTT}); the ''classical'' unbiased estimator for $C_\ell$ is given by $\frac{1}{2\ell+1}\sum_{m=-\ell}^\ell |a_{m\ell}|^2$ and for a Gaussian random field $f$ the study of its asymptotic properties is straightforward (see \cite{A, MP}).

From the perspective of statistical inference for processes on $\mathbb R^n$ there is a classical tool for estimation related to such quantities as the angular power spectrum, that is to say, to the covariance structure of a process, namely (empirical) quadratic variations, usually having the form $\sum_{i=1}^N (X_{(i+1)\Delta}-X_{i\Delta})^2$ for some process $X$. This tool is widely used for parameter estimation related to self-similar processes (see the monograph \cite{Tu}), but also for statistical inference for stochastic (partial) differential equations (see e.g. the overview \cite{C}). A practical advantage of such a statistics is that it allows to work with discrete observations and is simple to implement and to analyse. From the theoretical point of view, quadratic (or higher order) variations give access to structural changes of processes and (via their connection to the covariance structure) characterise their smoothness. One of the best known examples is the fractional Brownian motion whose quadratic variations famously change their limiting distribution at $H=\frac{3}{4}$ (see e.g. \cite{Tu}).

Quadratic variations of processes on a sphere have not been widely studied. In particular, there is no canonical discretisation. In the paper \cite{I} the author suggests several possible discretisation schemes and analyses empirical quadratic variations of the so-called spherical fractional Brownian motion. Apart from this work there are no studies of this subject known to the author. As to the Fourier components of an isotropic process, they depend on a finite number of random variables, their smoothness is only dependent on the spherical harmonics $Y_{m\ell}$ and the asymptotic properties of the corresponding quadratic variations are largely governed by the behaviour of the spherical harmonics, so the analysis seems to lie outside the realm of probability and be significantly different than for ''rougher'' processes. However, the results in probabilistic terms prove to be not only convenient but also useful.

In the present paper we consider quadratic variations of Gaussian isotropic random fields on the sphere defined with equidistant increments over a single geodesic line between two points and study their asymptotic properties. In particular, we show limit theorems in different asymptotic regimes using Malliavin-Stein calculus as the main tool in our proofs. More precisely, we prove a noncentral limit theorem for quadratic variations of $f_\ell$ for a fixed frequency $\ell$ and a quantitative central limit theorem for growing frequency $\ell\to \infty$. An interesting result is that, while the speed of convergence changes depending on the regime (i.e. how fast $\ell$ grows compared to the discretisation refinement $N$), the third order speed of convergence (i.e. the total variation, Wasserstein and Kolmogorov distances to the standard normal distribution) can be bounded by $\frac{1}{\log N}$ (up to a constant). We use these results to define angular power spectrum estimators and compare them to the classical estimator defined above. In particular, we find out that the rate of convergence of these estimators to the normal distribution in terms of $\ell$ and $N$ is also independent of the regime.

Under certain additional assumptions we also prove a quantitative central limit theorem for quadratic variations of the field $f$ and show how this can be applied for Hurst parameter estimation for a fractional Brownian motion defined over the Hilbert space $L^2(\mathbb S^2)$, which is constructed differently from the field in \cite{I}. 

Our paper is structured as follows: in Sections \ref{background} and \ref{MC} we introduce the main objects and some basic notions and theorems from Malliavin calculus used in this article. In Section \ref{fixed} we focus on the study of asymptotics of quadratic variations for a fixed frequency component when the number of increments tends to infinity. In Section \ref{increasing} we derive asymptotic results in the increasing frequency regime, depending on whether the discretisation increases faster, slower or with the same speed as the frequency. Using results from this section, we derive a quantitative central limit theorem for quadratic variations of the whole field $f$ and discuss an exemplary application in parameter estimation. Finally, in Section \ref{estimation} we introduce several estimators of the angular power spectrum, discuss their asymptotic properties and compare them to the classical estimator.
\section{Background and notation}\label{background}
As mentioned in the introduction, we consider an isotropic Gaussian random field
$$f(x)=\sum_{\ell=1}^\infty f_\ell(x)=\sum_{\ell=1}^\infty \sum_{m=-\ell}^\ell a_{m\ell}Y_{m\ell}(x), $$
where $a_{m\ell}$ are such that $\mathbb E [a_{m\ell}\overline{a_{m'\ell'}}]=C_\ell\delta_{\ell\ell'}\delta_{mm'}$. Note that the angular power spectrum $(C_\ell)_{\ell\in\mathbb N}$ is such that $\sum_{\ell\in\mathbb N} C_\ell \frac{2\ell +1}{4\pi}=\mathbb E f^2<\infty$ and $f_\ell$ is a real Gaussian random field (see \cite{MP}).
We have, moreover, by properties of the spherical harmonics
\[\mathbb E [f_\ell (x)f_\ell (y)]=C_\ell \frac{2\ell +1}{4\pi}P_\ell (\cos d(x,y)),\]
where $d(x,y)$ is the spherical geodesic distance between $x$ and $y$ and $P_\ell$ are the Legendre polynomials defined by Rodrigues' formula
\[P_\ell(t):=\frac{1}{2^\ell \ell !}\frac{d^\ell}{d t^\ell}(t^2-1)^\ell.\]
Our main object of study is the sequence of (empirical) quadratic variations defined along a line $\{f_\ell((0,a))| a\in (0,\,\pi/2)\}$ (where we use spherical coordinates) as
\[V_{N,\,\ell}:=\sum_{i=1}^N \left(f_\ell \left(\frac{i+1}{N}\frac{\pi}{2}\right)- f_\ell \left(\frac{i}{N}\frac{\pi}{2}\right)\right)^2,\]
where the notation $f_\ell (a)$ stands for $f_\ell ((0,a))$. Note that due to isotropy of the process and its Fourier components the zero and the descent direction can be chosen arbitrarily, that is, for statistical applications any line of length $\frac{\pi}{2}$ can be considered. The geodesic distance between two points on such a line is, of course, given by the difference of the second arguments. In a complete analogy, we also define empirical quadratic variations for the whole field $f$ as
\[V_{N}:=\sum_{i=1}^N \left(f \left(\frac{i+1}{N}\frac{\pi}{2}\right)- f \left(\frac{i}{N}\frac{\pi}{2}\right)\right)^2,\]
where the notation $f (a)$ stands for $f ((0,a))$.

In the course of the paper we use the symbol $\sim$ to denote asymptotic equality (i.e. the ratio is tending to one), the symbol $\sim_c$ to denote asymptotic equality up to a constant, and the symbol $\lesssim$ to denote that the left side is asymptotically less or equal to the right side up to a constant (i.e. the ratio is asymptotically bounded by a constant).

\section{Basics of Malliavin calculus}\label{MC}
In this chapter we recapitulate some basic definitions and theorems from Malliavin calculus that will be of use for our study. The main reference for details on those and other results is the monographs \cite{NP} or \cite{N}.

Let $(\mathcal H, \langle\cdot ,\, \cdot \rangle_{\mathcal H})$ be a real separable Hilbert space and $(B(\varphi), \,\varphi\in\mathbb H)$ an isonormal Gaussian process on some probability space $(\Omega, \mathcal F, P)$, that is, a centred Gaussian family such that $\mathbb E [B(\varphi)B(\psi)]=\langle\varphi ,\, \psi \rangle_{\mathcal H}$ for all $\varphi, \, \psi\in\mathcal H$.

In the special case of  $\mathcal{H}=L^2(A,\mathcal{A}, \mu)$ over a Polish space $A$ with the associated $\sigma$-field $\mathcal{A}$ and a positive $\sigma$-finite and non-atomic measure $\mu$ one can define an isonormal Gaussian process with respect to the inner product $\langle g,h\rangle_{\mathcal H}=\int_A g(a)h(a)d\mu (a)$ as the Wiener-It\^o integral
\[W(h)=\int_A h(a)dW(a)\]
with respect to a Gaussian family $W=\{W(B): B\in\mathcal A , \mu(B)<\infty \}$ such that $\mathbb E [W(A)W(B)]=\mu (A\cap B)$ for $A$, $B$ of finite measure.

For a fixed $q\geq 1$ Hermite polynomials are defined as
\begin{equation*}
H_{q}(x)=(-1)^{q} \exp \left( \frac{x^{2}}{2} \right) \frac{{\mathrm{d}}^{q}%
}{{\mathrm{d}x}^{q}}\left( \exp \left(
-\frac{x^{2}}{2}\right)\right),\; x\in \mathbb{R}.
\end{equation*}
Based on this definition one can define orthogonal subspaces of $L^2(\Omega,\mathcal F, P)$ called Wiener chaoses as closed linear spans of the random variables $H_{q}(W(\varphi))$,
$\varphi\in{\mathcal{H}}$ with $\Vert\varphi\Vert_{{\mathcal{H}}}=1$ for a given $q\geq 1$. There exists an isometry between the  Hilbert space ${\mathcal{H}}^{\odot q}$ (meaning the symmetric tensor product) equipped with the norm
$\frac{1}{\sqrt{q!}}\Vert\cdot\Vert_{{\mathcal{H}}^{\otimes q}}$ and the $q$th Wiener chaos. We denote it by $I_q(\cdot)$ and call it the $q$th multiple stochastic integral with respect to $W$.

The orthogonality and isometry properties can be written in the following way: for $p,\;q\geq
1$,\;$f\in{{\mathcal{H}}^{\otimes p}}$ and
$g\in{{\mathcal{H}}^{\otimes q}}$
\begin{equation*}
\mathbb{E}\Big(I_{p}(f) I_{q}(g) \Big)=
\begin{cases}
q! \langle \tilde{f},\tilde{g}
\rangle _{{\mathcal{H}}^{\otimes q}} & \mbox{if}\;p=q,\\
 0 & \mbox{otherwise},
\end{cases}
\end{equation*}
where $\tilde{f} $ denotes the canonical symmetrization of $f$.

We will further make use of the following product formula for multiple integrals: if  $f\in{{\mathcal{H}}^{\odot p}}$ and
$g\in{{\mathcal{H}}^{\odot q}}$, then 
\begin{eqnarray}\label{product}
I_{p}(f) I_{q}(g)&=& \sum_{r=0}^{p \wedge q} r! \binom{p}{r}\binom{q}{r}I_{p+q-2r}\left(f\tilde{\otimes}_{r}g\right).
\end{eqnarray}
Another important property of multiple integrals is hypercontractivity:  if $F=  I_{q}(f) $ with $f \in \mathcal{H} ^{\otimes q}$ for any $q\geq 0$, then
\begin{equation}
\label{hyper}
\mathbf{E}\vert F \vert ^{p} \leq C_{p} \left( \mathbf{E}F ^{2} \right) ^{\frac{p}{2}}.
\end{equation}
for every $p\geq 2$.

Denote by $D$ the Malliavin derivative operator that acts on cylindrical random variables of the form $F=g(W(\varphi
_{1}),\ldots,W(\varphi_{n}))$, where $n\geq 1$,
$g:\mathbb{R}^n\rightarrow\mathbb{R}$ is a smooth function with at most polynomially growing derivatives and $\varphi_{i} \in {{\mathcal{H}}}$. This derivative belongs to $L^2(\Omega,{\mathcal{H}})$ and it is defined as
\begin{equation*}
DF=\sum_{i=1}^{n}\frac{\partial g}{\partial x_{i}}(W(\varphi _{1}),
\ldots , W(\varphi_{n}))\varphi_{i}.
\end{equation*}
The derivative operator can be considered as inverse to the multiple integrals in the sense that for $q\geq 1$ and $f\in\mathcal H^{\odot q}$
\begin{equation}\label{derivM}
DI_q(f)=qI_{q-1}(f).
\end{equation}

Let us now quote the celebrated Fourth moment theorem (given in this form in \cite{NP}) that will  be the main tool for proving central limit theorems in this article.
\begin{theorem}\label{4MT}
Let $(F_N)_{N\in\mathbb N}=(I_q(f_N))_{N\in\mathbb N}$ for a fixed $q\geq 1$ and symmetric $f_N$ be such that
$$\mathbb E [F_N^2]\stackrel{N\to\infty}{\to}\sigma^2 >0.$$
Then $F_N$ converge in law to $Z\sim N(0,\,1)$ if and only if
$$\| DF_N \|_{\mathcal H}^2\stackrel{N\to\infty}{\to}q\sigma^2.$$
Moreover, for $d$ being either  Kolmogorov, total variation or Wasserstein distance, we have for a constant $C$
$$d(F_N, N(0,\,1))\leq C \sqrt{\Var \left(\frac{1}{q}\| DF_N \|_{\mathcal H}^2\right)}.$$
\end{theorem}

Now consider specifically the Hilbert space $\mathcal H :=L^2(\mathbb S^2,  d\sigma(x))$, where $d\sigma(x)$ denotes the Lebesgue measure on the sphere. This space can be associated with an isonormal Gaussian process $W$ in the manner indicated above such that $f_\ell (x)$ for any $x\in\mathbb S^2$ can be seen as a single integral over $W$, namely as
\[f_\ell (x)=\int \sqrt{C_\ell}\frac{2\ell +1}{4\pi}P_\ell (\langle x,\, y\rangle)dW(y).\]

The last result of this section is a noncentral limit theorem that is known in the literature for second Wiener chaos random variables over $L^2(\mathbb R^n)$, but whose proof can be easily exended for our Hilbert space of interest, $L^2(\mathbb S^2, d\sigma (x))$.
\begin{theorem}\label{NCLT}
Let $(F_N)_{N\in\mathbb N}$ be a normalised sequence in the second Wiener chaos over $L^2(\mathbb S^2, d\sigma (x))$, i.e. $F_N=I_2(f_N)$ with symmetric $f_N$ of norm one. Then its cumulants of order $p\geq 2$ are given by the formula
\[\kappa_p(F_N)=2^{p-1}(p-1)!\int_{(\mathbb S^2)^p}f_N(x_1,x_2)\dots f_N(x_{p-1},x_p)f_N(x_p,x_1)d\sigma(x_1)\dots d\sigma(x_p).\]
Assume that all cumulants converge for $N\to\infty$. Then the sequence $F_N$ converges weakly to a random variable $F$ characterised by its cumulants $\kappa_p(F)=\lim_{N\to\infty}\kappa_p(F_N)$, $\kappa_1(F)=0$ and $\kappa_2(F)=1$.
\end{theorem}
\begin{proof}
The first part of the proof is a completely analogous adaptation of the proof of Proposition 7.2 in \cite{NfBm}, while the second part follows by the hypercontractivity property of multiple Wiener integrals (cf. Corollary 2.8.14 in \cite{NP}) and a limiting argument for moments analogous to the argument of Theorem 5 in \cite{SST}.
\end{proof}
\section{Regime $N\to\infty$ for a fixed frequency $\ell$}\label{fixed}
In this chapter we study the behaviour of quadratic variations $V_{N,\ell}$ as $N$ tends to infinity for a fixed frequency $\ell$. The first step in establishing a limit theorem is determining a proper normalisation, that is, computing its expectation and variance.

We start thus by calculating the expectation:
\begin{eqnarray*}
&&\mathbb E [V_{N,\ell}]=\mathbb E \left[\sum_{i=1}^N \left(f_\ell \left(\frac{i+1}{N}\frac{\pi}{2}\right)- f_\ell \left(\frac{i}{N}\frac{\pi}{2}\right)\right)^2\right]  \\
&&=\sum_{i=1}^N \left(\mathbb E \left[f_\ell \left(\frac{i+1}{N}\frac{\pi}{2}\right)^2\right]+ \mathbb E \left[f_\ell \left(\frac{i}{N}\frac{\pi}{2}\right)^2\right]-2\mathbb E \left[f_\ell \left(\frac{i+1}{N}\frac{\pi}{2}\right)(f_\ell \left(\frac{i}{N}\frac{\pi}{2}\right)\right]\right)\\
&&=\sum_{i=1}^N 2 C_\ell \frac{2\ell + 1}{4\pi}(P_\ell (1)-P_{\ell}\left(\cos \frac{\pi}{2N}\right))\\
&&\sim 2N  C_\ell \frac{2\ell + 1}{4\pi} \left(\frac{\pi}{2 N}\right)^2\frac{\ell (\ell +1)}{2}\\
&& = \frac{\pi}{16}(2\ell +1)\ell (\ell +1)C_\ell \frac{1}{N},
\end{eqnarray*}
where the approximation step is due to Taylor's approximation, since $P'_\ell(1)=\frac{\ell (\ell +1)}{2}$.\\

Now let us move on to variance. First, recall that we have
\begin{eqnarray*}
\Var(\sum_{i=1}^NX_i^2)=\sum_{i,j=1}^N \mathbb E[X_i^2X_j^2]-(\sum_{i=1}^N\mathbb E [X_i^2])^2=2\sum_{i,j=1}^N\mathbb E [X_iX_j]^2 
\end{eqnarray*}
for Gaussian $X_i$. Thus, we need to compute
\[\sum_{i,j=1}^N \mathbb E \left[ \left(f_\ell \left(\frac{i+1}{N}\frac{\pi}{2}\right)- f_\ell \left(\frac{i}{N}\frac{\pi}{2}\right)\right) \left(f_\ell \left(\frac{j+1}{N}\frac{\pi}{2}\right)- f_\ell \left(\frac{j}{N}\frac{\pi}{2}\right)\right)\right]^2.\]
We have
\begin{eqnarray*}
&&\mathbb E \left[ \left(f_\ell \left(\frac{i+1}{N}\frac{\pi}{2}\right)- f_\ell \left(\frac{i}{N}\frac{\pi}{2}\right)\right) \left(f_\ell \left(\frac{j+1}{N}\frac{\pi}{2}\right)- f_\ell \left(\frac{j}{N}\frac{\pi}{2}\right)\right)\right]\\
&&=C_\ell \frac{2\ell +1}{4\pi}\left(2 P_\ell \left(\cos \frac{|i-j|}{N}\frac{\pi}{2}\right)-  P_\ell \left(\cos \frac{|i-j-1|}{N}\frac{\pi}{2}\right) -  P_\ell \left(\cos \frac{|i-j+1|}{N}\frac{\pi}{2}\right)\right),
\end{eqnarray*}
and hence,
\begin{eqnarray*}
&&\sum_{i,j=1}^N \mathbb E \left[ \left(f_\ell \left(\frac{i+1}{N}\frac{\pi}{2}\right)- f_\ell \left(\frac{i}{N}\frac{\pi}{2}\right)\right) \left(f_\ell \left(\frac{j+1}{N}\frac{\pi}{2}\right)- f_\ell \left(\frac{j}{N}\frac{\pi}{2}\right)\right)\right]^2\\
&&=C_\ell^2 \left(\frac{2\ell +1}{4\pi}\right)^2 \sum_{k=1}^N (N-k) \left(2 P_\ell \left(\cos \frac{k}{N}\frac{\pi}{2}\right)-  P_\ell \left(\cos \frac{k-1}{N}\frac{\pi}{2}\right) -  P_\ell \left(\cos \frac{k+1}{N}\frac{\pi}{2}\right)\right)^2.
\end{eqnarray*}
For $N\to \infty$ we can write
\begin{eqnarray}
&& 2 P_\ell \left(\cos \frac{k}{N}\frac{\pi}{2}\right)-  P_\ell \left(\cos \frac{k-1}{N}\frac{\pi}{2}\right) -  P_\ell \left(\cos \frac{k+1}{N}\frac{\pi}{2}\right)\label{approx}\\
&& \sim - \frac{1}{N^2}\frac{d^2}{dx^2}P_\ell(\cos (x\frac{\pi}{2}))\Big|_{\frac{k}{N}} = \frac{1}{N^2}\frac{\pi^2}{4} \left(\ell(\ell+1)P_\ell\left(\cos\frac{k}{N}\frac{\pi}{2} \right)- P'_\ell\left(\cos \frac{k}{N}\frac{\pi}{2}\right) \cos \left(\frac{k}{N}\frac{\pi}{2}\right)\right)\nonumber
\end{eqnarray}
and obtain
\begin{eqnarray*}
&&\sum_{i,j=1}^N \mathbb E \left[ \left(f_\ell \left(\frac{i+1}{N}\frac{\pi}{2}\right)- f_\ell \left(\frac{i}{N}\frac{\pi}{2}\right)\right) \left(f_\ell \left(\frac{j+1}{N}\frac{\pi}{2}\right)- f_\ell \left(\frac{j}{N}\frac{\pi}{2}\right)\right)\right]^2\\
&&\sim C_\ell^2 \left(\frac{2\ell +1}{4\pi}\right)^2\frac{\pi^4}{16N^4} \sum_{k=1}^N (N-k)  \left(\ell(\ell+1)P_\ell\left(\cos\frac{k}{N}\frac{\pi}{2} \right)- P'_\ell\left(\cos \frac{k}{N}\frac{\pi}{2}\right) \cos \left(\frac{k}{N}\frac{\pi}{2}\right)\right)^2\\
&&\sim  C_\ell^2 \left(\frac{2\ell +1}{4\pi}\right)^2\frac{\pi^4}{16N^4} N^2 \sum_{k=1}^N \left(\ell(\ell+1)P_\ell\left(\cos\frac{k}{N}\frac{\pi}{2} \right)- P'_\ell\left(\cos \frac{k}{N}\frac{\pi}{2}\right) \cos \left(\frac{k}{N}\frac{\pi}{2}\right)\right)^2 \frac{1}{N}\\
&&\sim  C_\ell^2 \left(\frac{2\ell +1}{4\pi}\right)^2\frac{\pi^4}{16N^4} N^2 \int_0^1  \left(\ell(\ell+1)P_\ell\left(\cos x \frac{\pi}{2} \right)- P'_\ell\left(\cos x\frac{\pi}{2}\right) \cos \left(x\frac{\pi}{2}\right)\right)^2 dx\\
&&\sim K_\ell C_\ell^2 \frac{1}{N^2},
\end{eqnarray*}
where $K_\ell$ is a (known) constant depending on $\ell$.

The sequence $(V_{N,\ell})_{N\in\mathbb N}$ does not obey a central limit theorem even after a proper normalisation. Instead, we can prove the following noncentral limit theorem.
\begin{theorem}\label{NCLTSphere}
The sequence  $$F_{N}:=\frac{V_{N,\ell}-\mathbb E V_{N,\ell}}{\sqrt{\Var(V_{N,\ell})}}$$ converges in distribution to a random variable $F$ with cumulants
\begin{eqnarray*}
&&\kappa_p(F)=\frac{2^{p-1}(p-1)!}{2^{p/2}\left(\int_0^1  \left(\ell(\ell+1)P_\ell\left(\cos x \frac{\pi}{2} \right)- P'_\ell\left(\cos x\frac{\pi}{2}\right) \cos \left(x\frac{\pi}{2}\right)\right)^2 dx\right)^{p/2}}\\
&&\times \int_{[0,1]^p}\left(\ell(\ell+1)P_\ell\left(\cos (|x_1-x_2|\frac{\pi}{2}) \right)- P'_\ell\left(\cos (|x_1-x_2|\frac{\pi}{2})\right) \cos \left(|x_1-x_2|\frac{\pi}{2}\right)\right)\times\\
&&\qquad \dots\times  \left(\ell(\ell+1)P_\ell\left(\cos(|x_p-x_1|\frac{\pi}{2}) \right)- P'_\ell\left(\cos (|x_p-x_1|\frac{\pi}{2})\right) \cos \left(|x_p-x_1|\frac{\pi}{2}\right)\right)dx_1\dots dx_p
\end{eqnarray*}
for $p > 3$, $\kappa_1(F)=0$ and $\kappa_2(F)=1$.
\end{theorem}
\begin{proof}
 Recall that we have
\[f_\ell(x)=\int \sqrt{C_\ell}\frac{2\ell+1}{4\pi}P_\ell (\langle x,\,y\rangle)dW(y)\]
and thus
\[f_\ell \left(\frac{i+1}{N}\frac{\pi}{2}\right)- f_\ell \left(\frac{i}{N}\frac{\pi}{2}\right)=I_1\left(\sqrt{C_\ell}\frac{2\ell+1}{4\pi} \left(P_\ell (\langle \frac{i+1}{N}\frac{\pi}{2},\cdot \rangle)-P_\ell (\langle \frac{i}{N}\frac{\pi}{2},\cdot \rangle) \right)\right)=:I_1(h^i_N(\cdot)).\]
Consequently, $F_N=I_2(g_N)$, where
$$g_N=\frac{1}{\sqrt{\Var(V_{N,\ell})}}\sum_{i=1}^N (h^i_N)^{\otimes 2}.$$
By Theorem \ref{NCLT} it suffices to show convergence of cumulants of $F_N$ of order higher than $2$. We calculate
\begin{eqnarray*}
&& \kappa_p(F_N)=2^{p-1}(p-1)!\int_{(\mathbb S^2)^p}g_N(x_1,x_2)\dots g_N(x_{p-1},x_p)g_N(x_p,x_1)d\sigma(x_1)\dots d\sigma(x_p)\\
&& = \frac{2^{p-1}(p-1)!}{\Var (V_{N,\ell})^{p/2}}\sum_{i_1,\dots , i_p=1}^N \int_{\mathbb S^2}h^{i_1}_N(x) h^{i_2}_N(x)d\sigma (x)\dots  \int_{\mathbb S^2}h^{i_{p-1}}_N(x) h^{i_p}_N(x)d\sigma (x) \int_{\mathbb S^2}h^{i_p}_N(x) h^{i_1}_N(x)d\sigma (x)\\
&&=\frac{2^{p-1}(p-1)!}{\Var (V_{N,\ell})^{p/2}}\sum_{i_1,\dots , i_p=1}^N \mathbb E [I_1(h^{i_1}_N)I_1(h^{i_2}_N)]\dots  \mathbb E [I_1(h^{i_{p-1}}_N)I_1(h^{i_p}_N)] \mathbb E [I_1(h^{i_p}_N)I_1(h^{i_1}_N)]\\
&&\sim \frac{2^{p-1}(p-1)! N^p}{(2 K_\ell)^{p/2}C_\ell^p}\sum_{i_1,\dots , i_p=1}^N \mathbb E [I_1(h^{i_1}_N)I_1(h^{i_2}_N)]\dots  \mathbb E [I_1(h^{i_{p-1}}_N)I_1(h^{i_p}_N)] \mathbb E [I_1(h^{i_p}_N)I_1(h^{i_1}_N)],
\end{eqnarray*}
inserting the asymptotic expression for variance. Using the approximations for $ \mathbb E [I_1(h^{i}_N)I_1(h^{j}_N)]$ computed in \eqref{approx} we can write
\begin{eqnarray*}
&& \kappa_p(F_N)\sim \frac{2^{p-1}(p-1)! N^p}{(2 K_\ell)^{p/2}C_\ell^p}C_\ell^p \left(\frac{2l+1}{4\pi}\right)^p\\
&& \qquad\sum_{i_1,\dots , i_p=1}^N\left(2 P_\ell \left(\cos \frac{|i_1-i_2|}{N}\frac{\pi}{2}\right)-  P_\ell \left(\cos \frac{|i_1-i_2-1|}{N}\frac{\pi}{2}\right) -  P_\ell \left(\cos \frac{|i_1-i_2+1|}{N}\frac{\pi}{2}\right)\right)\times\\
&&\qquad \dots \times \left(2 P_\ell \left(\cos \frac{|i_p-i_1|}{N}\frac{\pi}{2}\right)-  P_\ell \left(\cos \frac{|i_p-i_1-1|}{N}\frac{\pi}{2}\right) -  P_\ell \left(\cos \frac{|i_p-i_1+1|}{N}\frac{\pi}{2}\right)\right)\\
&&\sim \frac{2^{p-1}(p-1)! N^p}{(2 K_\ell)^{p/2}C_\ell^p}C_\ell^p \left(\frac{2l+1}{4\pi}\right)^p\frac{1}{N^{2p}}\frac{\pi^{2p}}{4^p}\\
&&\qquad \sum_{i_1,\dots , i_p=1}^N \left(\ell(\ell+1)P_\ell\left(\cos\frac{|i_1-i_2|}{N}\frac{\pi}{2} \right)- P'_\ell\left(\cos \frac{|i_1-i_2|}{N}\frac{\pi}{2}\right) \cos \left(\frac{|i_1-i_2|}{N}\frac{\pi}{2}\right)\right)\times\\
&&\qquad \dots\times  \left(\ell(\ell+1)P_\ell\left(\cos\frac{|i_p-i_1|}{N}\frac{\pi}{2} \right)- P'_\ell\left(\cos \frac{|i_p-i_1|}{N}\frac{\pi}{2}\right) \cos \left(\frac{|i_p-i_1|}{N}\frac{\pi}{2}\right)\right).
\end{eqnarray*}
Plugging in the expression for $K_\ell$, we obtain that for $N\to\infty$ this expression tends to
\begin{eqnarray*}
&&\frac{2^{p-1}(p-1)!}{2^{p/2}\left(\int_0^1  \left(\ell(\ell+1)P_\ell\left(\cos x \frac{\pi}{2} \right)- P'_\ell\left(\cos x\frac{\pi}{2}\right) \cos \left(x\frac{\pi}{2}\right)\right)^2 dx\right)^{p/2}}\\
&&\times \int_{[0,1]^p}\left(\ell(\ell+1)P_\ell\left(\cos (|x_1-x_2|\frac{\pi}{2}) \right)- P'_\ell\left(\cos (|x_1-x_2|\frac{\pi}{2})\right) \cos \left(|x_1-x_2|\frac{\pi}{2}\right)\right)\times\\
&&\qquad \dots\times  \left(\ell(\ell+1)P_\ell\left(\cos(|x_p-x_1|\frac{\pi}{2}) \right)- P'_\ell\left(\cos (|x_p-x_1|\frac{\pi}{2})\right) \cos \left(|x_p-x_1|\frac{\pi}{2}\right)\right)dx_1\dots dx_p,
\end{eqnarray*}
which is obviously finite.
\end{proof}

\section{Regime $\ell\to\infty$ (increasing frequency)}\label{increasing}
The situation changes drastically if $\ell$ also tends to infinity. In particular, in this chapter we will show a CLT for $\ell = \ell(N)\to \infty$ ad $N\to\infty$. A remarkable fact is that the normalisation depends on the relative speed of convergence of $\ell$ with respect to $N$; one can distinguish three different regimes.

Since $\ell$ is not a constant anymore, we cannot proceed in the same way as before, or at least need different arguments to do so. Let us instead recall some standard asymptotic results of Hilb's type for Legendre polynomials mentioned in \cite{CM} and in \cite{W} respectively.
\begin{prop}\label{P-asympt}
For $0<C<\psi <N\frac{\pi}{2}$ the following expansion holds:
\begin{eqnarray*}
	P_\ell \left(\cos\frac{\psi}{N}\right)&=&\sqrt{\frac{\psi / N}{\sin\frac{\psi}{N}}}J_0 \left(\left(\ell + \frac{1}{2}\right)\frac{\psi}{N}\right)+o(\ell^{-3/2}\sqrt{\psi/N}),
\end{eqnarray*}
where $J_0$ is a Bessel function. If additionally $\ell = \ell(N)$ is such that $\ell/N \to \infty$ or $\ell /N \to c\neq 0$ for $N\to\infty$, then we can use the asymptotics $J_0(x)\sim \frac{1}{\sqrt x}$ to derive
\begin{eqnarray*}
	P_\ell \left(\cos\frac{\psi}{N}\right)&=&\sqrt{\frac{2}{\pi \ell \sin\frac{\psi}{N}}}\left(\sin\left(\psi + \frac{\pi}{4}\right)+O\left(\frac{1}{\psi}\right)\right).
\end{eqnarray*}
\end{prop}

As in the previous chapter, we start by computing the expectation:
\begin{eqnarray*}
&&\mathbb E \left[\sum_{i=1}^N \left(f_\ell \left(\frac{i+1}{N}\frac{\pi}{2}\right)- f_\ell \left(\frac{i}{N}\frac{\pi}{2}\right)\right)^2\right]  \\
&&=\sum_{i=1}^N \left(\mathbb E \left[f_\ell \left(\frac{i+1}{N}\frac{\pi}{2}\right)^2\right]+ \mathbb E \left[f_\ell \left(\frac{i}{N}\frac{\pi}{2}\right)^2\right]-2\mathbb E \left[f_\ell \left(\frac{i+1}{N}\frac{\pi}{2}\right)(f_\ell \left(\frac{i}{N}\frac{\pi}{2}\right)\right]\right)\\
&&=\sum_{i=1}^N 2 C_\ell \frac{2\ell + 1}{4\pi}(P_\ell (1)-P_{\ell}\left(\cos \frac{\pi}{2N}\right))\\
&&=2N C_\ell \frac{2\ell + 1}{4\pi}(1-P_{\ell}\left(\cos\frac{\pi}{2N}\right))\\
&&\sim \begin{cases} 2 N C_\ell  \frac{2\ell + 1}{4\pi} & \text{ if } N/\ell(N)\to 0, \\ 2 N C_\ell  \frac{2\ell + 1}{4\pi}(1-J_0(\frac{\pi}{2}c))& \text{ if } \ell(N)/ N\to c,\\   2 N C_\ell  \frac{2\ell + 1}{4\pi} \frac{\pi^2}{8}\frac{\ell^2}{N^2}& \text{ if } \ell(N)/N\to 0.\end{cases}
\end{eqnarray*}
The first asymptotic equality is readily seen using the estimate 22.14.9 from \cite{AS}:
\[\left|P_\ell \left(\cos \left(\frac{\pi}{2N}\right)\right)\right|\leq\sqrt{\frac{2}{\pi \ell}}\frac{1}{\sqrt[4]{1-\cos^2\left(\frac{\pi}{2N}\right)}}=\sqrt{\frac{2}{\pi \ell \sin\left(\frac{\pi}{2N}\right)}}\stackrel{N\to\infty}{\to}0.\]
The asymptotics for $\ell/N\to c$ is given by the Mehler-Rayleigh formula (see e.g. \cite{BT}). Finally, the asymptotics in case $\ell /N\to 0$ is derived by Taylor's approximation applied to $J_0$, since 
\begin{eqnarray*}
	P_\ell \left(\cos\frac{\pi}{2N}\right)&\sim&J_0 \left(\ell \frac{\pi}{2N}\right)
\end{eqnarray*}
by Proposition \ref{P-asympt}.

We proceed the by determining the variance. Recall the formula
\begin{eqnarray*}
\Var(\sum_{i=1}^NX_i^2)=2\sum_{i,j=1}^N\mathbb E [X_iX_j]^2 
\end{eqnarray*}
for Gaussian $X_i$. Thus, as in the previous chapter, we need to calculate
\[\sum_{i,j=1}^N \mathbb E \left[ \left(f_\ell \left(\frac{i+1}{N}\frac{\pi}{2}\right)- f_\ell \left(\frac{i}{N}\frac{\pi}{2}\right)\right) \left(f_\ell \left(\frac{j+1}{N}\frac{\pi}{2}\right)- f_\ell \left(\frac{j}{N}\frac{\pi}{2}\right)\right)\right]^2.\]
We have
\begin{eqnarray*}
&&\mathbb E \left[ \left(f_\ell \left(\frac{i+1}{N}\frac{\pi}{2}\right)- f_\ell \left(\frac{i}{N}\frac{\pi}{2}\right)\right) \left(f_\ell \left(\frac{j+1}{N}\frac{\pi}{2}\right)- f_\ell \left(\frac{j}{N}\frac{\pi}{2}\right)\right)\right]\\
&&=C_\ell \frac{2\ell +1}{4\pi}\left(2 P_\ell \left(\cos \frac{|i-j|}{N}\frac{\pi}{2}\right)-  P_\ell \left(\cos \frac{|i-j-1|}{N}\frac{\pi}{2}\right) -  P_\ell \left(\cos \frac{|i-j+1|}{N}\frac{\pi}{2}\right)\right),
\end{eqnarray*}
and hence,
\begin{eqnarray*}
&&\sum_{i,j=1}^N \mathbb E \left[ \left(f_\ell \left(\frac{i+1}{N}\frac{\pi}{2}\right)- f_\ell \left(\frac{i}{N}\frac{\pi}{2}\right)\right) \left(f_\ell \left(\frac{j+1}{N}\frac{\pi}{2}\right)- f_\ell \left(\frac{j}{N}\frac{\pi}{2}\right)\right)\right]^2\\
&&=C_\ell^2 \left(\frac{2\ell +1}{4\pi}\right)^2 \sum_{k=1}^N (N-k) \left(2 P_\ell \left(\cos \frac{k}{N}\frac{\pi}{2}\right)-  P_\ell \left(\cos \frac{k-1}{N}\frac{\pi}{2}\right) -  P_\ell \left(\cos \frac{k+1}{N}\frac{\pi}{2}\right)\right)^2.
\end{eqnarray*}

For $\ell$ such that $\ell/N \to 0$ we can, upon writing
\begin{eqnarray*}
	P_\ell \left(\cos\frac{\psi}{N}\right)&\sim&J_0 \left(\ell \frac{\psi}{N}\right)
\end{eqnarray*}
for $N\to\infty$, obtain (similarly to the constant $\ell$ case) by mean value theorem (that is, considering $\frac{1}{N\slash \ell}$ as increment)
\begin{eqnarray*}
&&2 P_\ell \left(\cos \frac{k}{N}\frac{\pi}{2}\right)-  P_\ell \left(\cos \frac{k-1}{N}\frac{\pi}{2}\right) -  P_\ell \left(\cos \frac{k+1}{N}\frac{\pi}{2}\right)\\
&&\sim -\frac{\pi^2}{8}\frac{\ell^2}{N^2}\left(J_0\left(\ell\frac{k}{N}\frac{\pi}{2}\right)- J_2\left(\ell\frac{k}{N}\frac{\pi}{2}\right)\right),
\end{eqnarray*}
which is for growing $k$ of order
\begin{eqnarray*}
-\frac{\pi^{3/2}}{4\sqrt{2}}\left(\frac{\ell}{N}\right)^{3/2}\frac{1}{\sqrt{k}}.
\end{eqnarray*}
Thus, we have for the variance
\begin{eqnarray*}
&&\sum_{i,j=1}^N \mathbb E \left[ \left(f_\ell \left(\frac{i+1}{N}\frac{\pi}{2}\right)- f_\ell \left(\frac{i}{N}\frac{\pi}{2}\right)\right) \left(f_\ell \left(\frac{j+1}{N}\frac{\pi}{2}\right)- f_\ell \left(\frac{j}{N}\frac{\pi}{2}\right)\right)\right]^2\\
&&\sim C_\ell^2 \left(\frac{2\ell +1}{4\pi}\right)^2 N \frac{\pi^3}{32}\frac{\ell^3}{N^3} \sum_{k=1}^N \frac{1}{k}\sim C_\ell^2 \left(\frac{2\ell +1}{4\pi}\right)^2 \frac{\pi^3}{32}\ell^3 N^{-2}\log N\\
&&\sim C_\ell^2\frac{\pi}{128} \ell^5 N^{-2}\log N.
\end{eqnarray*}
Let us now analyse the case  $\ell/N \to \infty$ or $\ell /N \to c\neq 0$ for $N\to\infty$. Note first that for $\psi =k\frac{\pi}{2}$ the terms $\sin\left(\psi + \frac{\pi}{4}\right)$ are constant (and equal $\frac{\sqrt 2}{2}$) up to a sign change. Moreover, among the two terms $P_\ell \left(\cos \frac{k-1}{N}\frac{\pi}{2}\right)$ and $ P_\ell \left(\cos \frac{k+1}{N}\frac{\pi}{2}\right)$ exactly one has the same sign as $P_\ell \left(\cos \frac{k}{N}\frac{\pi}{2}\right)$ in the asymptotics. Let it (without loss of generality) be the first term. Then (ignoring the second summand that is dominated by the first) we obtain, as $\ell\to\infty$ and $N\to\infty$,
\begin{eqnarray*}
&&2 P_\ell \left(\cos \frac{k}{N}\frac{\pi}{2}\right)-  P_\ell \left(\cos \frac{k-1}{N}\frac{\pi}{2}\right) -  P_\ell \left(\cos \frac{k+1}{N}\frac{\pi}{2}\right)\\
&&\sim \sqrt\frac{2}{\pi \ell}\frac{\sqrt 2}{2}\left(2\sqrt\frac{1}{\sin \left(\frac{k}{N}\frac{\pi}{2}\right)}-\sqrt\frac{1}{\sin \left(\frac{k-1}{N}\frac{\pi}{2}\right)}+ \sqrt\frac{1}{\sin \left(\frac{k+1}{N}\frac{\pi}{2}\right)}\right)\\
&&\sim \sqrt\frac{2}{\pi \ell}\frac{\sqrt 2}{2}\left(-\frac{1}{N^2}\left(\frac{1}{2}\sqrt\frac{1}{\sin \left(\frac{k}{N}\frac{\pi}{2}\right)}-\frac{3}{4}\frac{\cos^2 \left(\frac{k}{N}\frac{\pi}{2}\right)}{\sin ^{5/2} \left(\frac{k}{N}\frac{\pi}{2}\right)}\right)+ 2\sqrt\frac{1}{\sin \left(\frac{k+1}{N}\frac{\pi}{2}\right)}\right)\\
&&\sim \sqrt\frac{2}{\pi \ell}\frac{\sqrt 2}{2}2\sqrt\frac{1}{\sin \left(\frac{k+1}{N}\frac{\pi}{2}\right)}\sim\frac{2}{\sqrt{\pi \ell \sin \left(\frac{k}{N}\frac{\pi}{2}\right) }},
\end{eqnarray*}
since the last term dominates the sum.

We obtain for the variance:
\begin{eqnarray*}
&& C_\ell^2 \left(\frac{2\ell +1}{4\pi}\right)^2 \sum_{k=1}^N (N-k) \left(2 P_\ell \left(\cos \frac{k}{N}\frac{\pi}{2}\right)-  P_\ell \left(\cos \frac{k-1}{N}\frac{\pi}{2}\right) -  P_\ell \left(\cos \frac{k+1}{N}\frac{\pi}{2}\right)\right)^2\\
&&\sim  C_\ell^2 \left(\frac{2\ell +1}{4\pi}\right)^2 N \sum_{k=1}^N \frac{4}{\pi \ell \sin \left(\frac{k}{N}\frac{\pi}{2}\right)  } \sim  C_\ell^2 \left(\frac{2\ell +1}{4\pi}\right)^2 N \frac{4}{\ell \pi}\frac{2}{\pi} N\log(N)\\
&&\sim \frac{2}{\pi^4}C_\ell^2 \ell N^2 \log (N).
\end{eqnarray*}

Let us summarise our results. We have for $N\to\infty$ and $\ell = \ell(N)\to\infty$:
\[\Var(V_{N,\ell}) \sim \begin{cases} C C_\ell^2 \ell N^2 \log N & \text{ if } N/\ell(N)\to 0, \\ CC_\ell^2 N^3 \log N & \text{ if } \ell(N)/ N\to c,\\   CC_\ell^2 \ell^5 N^{-2}\log N & \text{ if } \ell(N)/N\to 0,\end{cases}\]
where the constant $C$ can vary depending on the regime. In fact, we can see that there are only two different regimes for the variance: plugging in $\ell=cN$ into either the first or the third expression yields, indeed, the expression in the second case.

\begin{remark}
The fact that the expectation and variance of quadratic variations exhibit a change in their behaviour might seem somewhat surprising. Thinking in terms of the behaviour of Legendre polynomials might explain this phenomenon. If $N$ grows slower than $\ell$, one can say heuristically that there are more oscillations of the Legendre polynomials appearing than there are discretisation points ''smoothing out'' these irregularities.
Remarkably, the growth of $\sqrt{\Var (V_{N,\ell})}$ is slower by a logarithmic factor than that of $\mathbb E [V_{N,\ell}]$. Heuristically, again, one can explain this by noticing that while the expectation only depends on local behaviour of $f_\ell$, that is, on covariances between neighbours, the computation of variance involves covariance terms of points that are far apart. The former correspond to a monotonous, non-oscillating part of $P_\ell$ (i.e. its values close to one) while the latter are described by $P_\ell$ on the interval $(0,1)$ and involve many oscillations responsible for the growth of quadratic variations.
\end{remark}

Now that we have established the appropriate normalisation, we prove a CLT for the normalised version of the empirical quadratic variations, namely for
\[F_N:=\frac{1}{\sqrt{v_N}}\sum_{i=1}^N \left(f_\ell \left(\frac{i+1}{N}\frac{\pi}{2}\right)- f_\ell \left(\frac{i}{N}\frac{\pi}{2}\right)\right)^2-\mathbb E \left(f_\ell \left(\frac{i+1}{N}\frac{\pi}{2}\right)- f_\ell \left(\frac{i}{N}\frac{\pi}{2}\right)\right)^2, \]
where
\[v_N=\Var \left(\sum_{i=1}^N \left(f_\ell \left(\frac{i+1}{N}\frac{\pi}{2}\right)- f_\ell \left(\frac{i}{N}\frac{\pi}{2}\right)\right)^2\right)\sim C C_\ell^2 \ell N^2 \log(N)\]
for some constant $C$. To this end, recall that we have
\[f_\ell(x)=\int \sqrt{C_\ell}\frac{2\ell+1}{4\pi}P_\ell (\langle x,\,y\rangle)dW(y)\]
and thus
\[f_\ell \left(\frac{i+1}{N}\frac{\pi}{2}\right)- f_\ell \left(\frac{i}{N}\frac{\pi}{2}\right)=I_1\left(\sqrt{C_\ell}\frac{2\ell+1}{4\pi} \left(P_\ell (\langle \frac{i+1}{N}\frac{\pi}{2},\cdot \rangle)-P_\ell (\langle \frac{i}{N}\frac{\pi}{2},\cdot \rangle) \right)\right)=:I_1(h^i_N(\cdot)).\]
Therefore, one can view $F_N$ as elements of the second Wiener chaos with respect to $W$. We are going to prove a quantitative CLT by means of the fourth moment theorem.

Let us formulate and prove the CLT for $F_N$.
\begin{theorem}\label{CLTfl}
As $N\to\infty$ we have
\[d(F_N,\,N)\lesssim \frac{1}{\log N},\]
where $N$ is the standard normal distribution.
\end{theorem}
\begin{proof}
By standard calculations (involving the product formula and the formula for Malliavin derivatives of multiple integrals) we have
\[\Var (\|DF_N\|_{\mathcal H}^2)=\frac{1}{v_N^2}\sum_{i,j, k, l=1}^N \langle h^i_N, h^j_N\rangle_{L^2(\mathbb S^2)} \langle h^k_N, h^l_N\rangle_{L^2(\mathbb S^2)} \langle h^i_N, h^k_N\rangle_{L^2(\mathbb S^2)} \langle h^j_N, h^l_N\rangle_{L^2(\mathbb S^2)}.\]
Note that it suffices to show convergence of this term to obtain a CLT.

Starting from here we need to distinguish between the different regimes.
For $\ell/N \to \infty$ or $\ell /N \to c\neq 0$ as $N\to\infty$ we have
\begin{eqnarray*}
&&\langle h^i_N, h^j_N\rangle_{L^2(\mathbb S^2)} = C_\ell \frac{2\ell+1}{4\pi }\left(2 P_\ell \left(\cos \frac{|i-j|}{N}\frac{\pi}{2}\right)-  P_\ell \left(\cos \frac{|i-j-1|}{N}\frac{\pi}{2}\right) -  P_\ell \left(\cos \frac{|i-j+1|}{N}\frac{\pi}{2}\right)\right)\\
&&\sim C C_\ell \ell \sqrt{\frac{N}{\ell |i-j|}},
\end{eqnarray*}
following the asymptotic computations from above. We thus obtain
\begin{eqnarray*}
&&\Var (\|DF_N\|_{\mathcal H}^2)\sim C \frac{1}{v_N^2} C_\ell^4 \ell^4 \frac{N^2}{\ell^2}\sum_{i,j, k, l=1}^N  \frac{1}{\sqrt{|i-j||k-l||i-k||j-l|}}\\
&&\sim \frac{1}{C_\ell^4 \ell^2 N^4 \log^2 N}C_\ell^4 \ell^2 N^2 \sum_{i,j, k, l=1}^N   \frac{1}{\sqrt{|i-j||k-l||i-k||j-l|}}\\
&&\sim \frac{1}{N^2\log^2 N}\sum_{i,j, k, l=1}^N \frac{1}{\sqrt{|i-j||k-l||i-k||j-l|}}.
\end{eqnarray*}
Meanwhile, for $\ell /N \to 0$ as $N\to\infty$ we have
\begin{eqnarray*}
&&\langle h^i_N, h^j_N\rangle_{L^2(\mathbb S^2)} = C_\ell \frac{2\ell+1}{4\pi }\left(2 P_\ell \left(\cos \frac{|i-j|}{N}\frac{\pi}{2}\right)-  P_\ell \left(\cos \frac{|i-j-1|}{N}\frac{\pi}{2}\right) -  P_\ell \left(\cos \frac{|i-j+1|}{N}\frac{\pi}{2}\right)\right)\\
&&\sim C C_\ell \ell \left(\frac{\ell}{N}\right)^{3/2}\sqrt{\frac{1}{|i-j|}},
\end{eqnarray*}
and therefore,
\begin{eqnarray*}
&&\Var (\|DF_N\|_{\mathcal H}^2)\sim C \frac{1}{v_N^2} C_\ell^4 \ell^4 \frac{\ell^6}{N^6}\sum_{i,j, k, l=1}^N  \frac{1}{\sqrt{|i-j||k-l||i-k||j-l|}}\\
&&\sim \frac{1}{C_\ell^4 \ell^10 N^{-4} \log^2 N}C_\ell^4 \ell^{10} N^{-6} \sum_{i,j, k, l=1}^N   \frac{1}{\sqrt{|i-j||k-l||i-k||j-l|}}\\
&&\sim \frac{1}{N^2\log^2 N}\sum_{i,j, k, l=1}^N \frac{1}{\sqrt{|i-j||k-l||i-k||j-l|}}.
\end{eqnarray*}
Finally, we use a standard argument (see e.g. proof of Theorem 7.3.1 in \cite{NP}) and write for $\varphi_N (k)=\frac{1}{\sqrt{k}}1_{\{k\leq N, k\neq 0\}}$
\begin{eqnarray*}
&&\Var (\|DF_N\|_{\mathcal H}^2)\lesssim \frac{1}{N^2\log^2 N} \sum_{i,k=1}^N (\varphi_N\ast\varphi_N)^2 (k-i)\\
&&\lesssim \frac{1}{N \log^2 N} \sum_{k\in\mathbb Z} (\varphi_N\ast\varphi_N)^2 (k) = \frac{1}{N \log^2 N} \|(\varphi_N\ast\varphi_N\|_{\ell^2(\mathbb Z)}^2\\
&&\leq  \frac{1}{N \log^2 N} \|\varphi_N\|^4_{\ell^{3\slash 4}(\mathbb Z)}\sim C\frac{1}{N \log^2 N} \left(\sum_{k=1}^N k^{-2\slash 3}\right)^3\sim C \frac{1}{\log^2 N},
\end{eqnarray*}
and the proof is finished.
\end{proof}
\section{Quadratic variations of $f$}\label{whole}
The computations in the previous chapter allow us to derive a CLT for the whole field $f$. In this chapter we investigate the behaviour of the sequence
\[V_{N}:=\sum_{i=1}^N \left(f \left(\frac{i+1}{N}\frac{\pi}{2}\right)- f \left(\frac{i}{N}\frac{\pi}{2}\right)\right)^2,\]
where the same abbreviating notation is used as for $V_{N,\ell}$. To be able to calculate asymptotic covariances, we need to impose a concrete assumption on the behaviour of $C_\ell$, namely
\begin{equation}
 C_\ell\sim_c \ell^{-2-\varepsilon}\tag{A}
\end{equation}
for some small $\varepsilon >0$. Note that this is a simplification of a standard assumption which reflects the behaviour of the angular power spectrum associated with the cosmic microwave background (see e.g. Condition 1 in \cite{CM} and the discussion following it). From the technical point of view, this assumption is not necessary and one might be able to repeat the below computations for a more general class of $C_\ell$.

The plan of action does not differ from the sections before. However, here our main tool is to divide the sum over $\ell$ into two parts and treat them separately. We start with the computation of the expectation.
\begin{eqnarray*}
&&\mathbb E \left[\sum_{i=1}^N \left(f \left(\frac{i+1}{N}\frac{\pi}{2}\right)- f \left(\frac{i}{N}\frac{\pi}{2}\right)\right)^2\right]  \\
&&=\sum_{i=1}^N \left(\mathbb E \left[f \left(\frac{i+1}{N}\frac{\pi}{2}\right)^2\right]+ \mathbb E \left[f \left(\frac{i}{N}\frac{\pi}{2}\right)^2\right]-2\mathbb E \left[f \left(\frac{i+1}{N}\frac{\pi}{2}\right)(f_\ell \left(\frac{i}{N}\frac{\pi}{2}\right)\right]\right)\\
&&=\sum_{i=1}^N 2\sum_{\ell=1}^\infty C_\ell \frac{2\ell + 1}{4\pi}(P_\ell (1)-P_{\ell}\left(\cos \frac{\pi}{2N}\right))\\
&&=2N \sum_{\ell=1}^{N^\alpha}C_\ell \frac{2\ell + 1}{4\pi}(1-P_{\ell}\left(\cos\frac{\pi}{2N}\right))+2N \sum_{\ell=N^\alpha + 1}^{\infty}C_\ell \frac{2\ell + 1}{4\pi}(1-P_{\ell}\left(\cos\frac{\pi}{2N}\right))
\end{eqnarray*}
for some $\alpha \in (0,1)$. The first of the two summands is asymptotically equal to
\[2N \sum_{\ell =1}^{N^\alpha}C_\ell \frac{2\ell+1}{4\pi}\frac{\pi^2}{8}\frac{\ell(\ell+1)}{N^2}\sim_c N^{2\alpha -\alpha\varepsilon-1}.\]
The second summand is bounded by
\[2N \sum_{\ell=N^\alpha + 1}^{\infty}C_\ell \frac{2\ell + 1}{4\pi}\sim_c N^{1-\alpha\varepsilon}.\]
For $\alpha$ close to $1$ the asymptotic value of the first summand approaches (from below) the bound of the second, i.e. the expectation is bounded by $N^{1-\varepsilon}$ (up to a constant). Moreover, this bound is attained by the partial sum between $\ell=N$ and $\ell=2N$. Thus, the expactation is of order $N^{1-\varepsilon}$.

Now we turn to variance. Analogously to the computations for $f_\ell$, we obtain:
\begin{eqnarray*}
\Var(V_N)\sim_c\sum_{\ell_1,\ell_2=1}^\infty C_{\ell_1}C_{\ell_2}\ell_1\ell_2\sum_{k=1}^N N (\ddelta P_{\ell_1}(k, N))  (\ddelta P_{\ell_2}(k, N)),
\end{eqnarray*}
where
\[\ddelta P_{\ell}(k, N)=2 P_\ell \left(\cos \frac{k}{N}\frac{\pi}{2}\right)-  P_\ell \left(\cos \frac{k-1}{N}\frac{\pi}{2}\right) -  P_\ell \left(\cos \frac{k+1}{N}\frac{\pi}{2}\right).\]
Note that by calculations in the previous section the bound
\[|\ddelta P_{\ell}(k, N)|\leq C\frac{1}{\sqrt{k}}\]
with an absolute constant $C$ holds for all regimes of convergence.

We proceed now as in the expectation computation and split the sum:
\begin{eqnarray*}
\sum_{\ell_1,\ell_2=1}^\infty = \sum_{\ell_1,\ell_2=1}^{N^\alpha} +2 \sum_{\ell_1=N^\alpha +1}^{\infty}\sum_{\ell_2=1}^{N^{\alpha}} +  \sum_{\ell_1,\ell_2=N^{\alpha}+1}^{\infty}.
\end{eqnarray*}
The first summand is of order
\[N^{-2}\log(N)\left(\sum_{\ell}^{N^\alpha} \ell^{-1-\varepsilon}\ell^{3\slash 2}\right)^2\sim_c \log(N)N^{-2+3\alpha-2\varepsilon\alpha},\]
which can be seen by using the asymptotics for $\ddelta P_{\ell}(k, N)$ proved in the previous chapter.
Note that here we can directly assume $\alpha \in (0,1]$, since there is no change in regime at $\alpha=1$ for variance. For the second summand we obtain the bound
\[\sum_{\ell_1=N^\alpha + 1}^\infty \ell_1^{-1-\varepsilon}\sum_{\ell_2=1}^{N^\alpha}\ell_2^{-1-\varepsilon+3\slash 2}N^{1-3\slash 2}\log N\sim_c N^{-2\alpha\varepsilon+\frac{3}{2}\alpha-\frac{1}{2}}\log N,\]
using the approximation of $\ddelta P_{\ell}(k, N)$ for $\ell_2$ and the coarse universal bound for $\ell_1$. With the same coarse estimate we can bound the third summand by
\[N\log N \left(\sum_{\ell =N^\alpha +1}^{\infty} \ell^{-1-\varepsilon}\right)^2\sim N^{1-2\alpha\varepsilon}\log(N).\]
As with expectation, the bounds and the asymptotic speed of the first summand all fall together for $\alpha=1$, yielding
\[\Var(V_N)\sim_c \log(N)N^{1-2\varepsilon}.\]
In order to prove a CLT for $f$, recall first that $f(x)=\sum_{\ell=1}^\infty f_\ell(x)$ holds in $L^2$-sense. Thus,
\[f(x)=I_1\left(\sum_{\ell=1}^\infty  \sqrt{C_\ell}\frac{2\ell+1}{4\pi}P_\ell (\langle x,\cdot \rangle)\right),\]
and the fourth moment theorem can be applied. We repeat the notation used for $f_\ell$ and write
\[F_N:=\frac{1}{\sqrt{v_N}}\sum_{i=1}^N \left(f \left(\frac{i+1}{N}\frac{\pi}{2}\right)- f \left(\frac{i}{N}\frac{\pi}{2}\right)\right)^2-\mathbb E \left(f \left(\frac{i+1}{N}\frac{\pi}{2}\right)- f \left(\frac{i}{N}\frac{\pi}{2}\right)\right)^2, \]
where
\[v_N=\Var \left(\sum_{i=1}^N \left(f \left(\frac{i+1}{N}\frac{\pi}{2}\right)- f \left(\frac{i}{N}\frac{\pi}{2}\right)\right)^2\right)\sim_c \log (N) N^{1-2\varepsilon}. \]
We further denote
\[\sqrt{C_\ell}\frac{2\ell+1}{4\pi} \left(P_\ell (\langle \frac{i+1}{N}\frac{\pi}{2},\cdot \rangle)-P_\ell (\langle \frac{i}{N}\frac{\pi}{2},\cdot \rangle) \right)=:h^i_{N,\ell}\]
and $$\sum_{\ell=1}^\infty h^i_{N,\ell}=:H^i_N,$$
such that
\[f \left(\frac{i+1}{N}\frac{\pi}{2}\right)- f \left(\frac{i}{N}\frac{\pi}{2}\right)=I_1(H^i_N(\cdot)).\]

\begin{theorem}\label{CLTf}
As $N\to\infty$ we have
\[d(F_N,\,N)\lesssim \frac{1}{\log N},\]
where $N$ is the standard normal distribution.
\end{theorem}
\begin{proof}
As in the proof of Theorem \ref{CLTfl}, we have
\[\Var (\|DF_N\|_{\mathcal H}^2)=\frac{1}{v_N^2}\sum_{i,j, k, l=1}^N \langle H^i_N, H^j_N\rangle_{L^2(\mathbb S^2)} \langle H^k_N, H^l_N\rangle_{L^2(\mathbb S^2)} \langle H^i_N, H^k_N\rangle_{L^2(\mathbb S^2)} \langle H^j_N, H^l_N\rangle_{L^2(\mathbb S^2)}.\]
Note first that by orthogonality we have
\[ \langle H^i_N, H^j_N\rangle_{L^2(\mathbb S^2)}=\sum_{\ell=1}^\infty \langle h^i_{N,\ell},\, h^i_{N,\ell}\rangle_{L^2(\mathbb S^2)},\]
and we can, again, split the sum into a part below $N^\alpha$ and a part above it. By usual calculation, using the asymptotics for $\ddelta P_\ell$ and its coarse estimate respectively, we can see that the first summand is of order
$$\frac{1}{\sqrt{|i-j|}}N^{-\frac{3}{2}+\frac{3}{2}\alpha-\alpha\varepsilon}$$
and the second is bounded by $\frac{1}{\sqrt{|i-j|}} N^{-\alpha\varepsilon}$. With arguments as before it follows that
\[\left|\langle H^i_N, H^j_N\rangle_{L^2(\mathbb S^2)}\right|\lesssim \frac{1}{\sqrt{|i-j|}} N^{-\varepsilon}.\]
Thus,
\begin{eqnarray*}
&&\Var (\|DF_N\|_{\mathcal H}^2)\lesssim \frac{1}{v_N^2}N^{-4\varepsilon}\sum_{i,j, k, l=1}^N \frac{1}{\sqrt{|i-j||k-l||i-k||j-l|}}\\
&&\sim_c \frac{1}{\log^2(N) N^{2-4\varepsilon}}N^{-4\varepsilon}\sum_{i,j, k, l=1}^N \frac{1}{\sqrt{|i-j||k-l||i-k||j-l|}}\sim_c \frac{1}{\log^2(N)}
\end{eqnarray*}
by the same closing argument as in Theorem \ref{CLTfl}.
\end{proof}
\begin{remark}
Clearly, the results in this section can be used for parameter estimation related to the covariance function. Indeed, the estimator $\hat{c}:=\frac{V_N}{N^{1-\varepsilon}}$ is an asymptotically unbiased estimator of the constant in front of the expectation with variance of order $\frac{\log N}{N}$. Thus, it is weakly consistent. Let us now consider as an example an $L^2(\mathbb S^2)$-valued fractional Brownian motion $B_t^H$ defined as
\[B_t^H(x):=\sum_{\ell = 0}^\infty\sum_{m=-\ell}^\ell B^H_{\ell m}(t)Y_{\ell m}(x)\]
for $x\in\mathbb S^2$, where $B^H_{\ell m}(t)$ are independent complex-valued fractional Brownian motions over $\mathbb R^+$ with Hurst index $H$, variances $A_\ell$ and $\operatorname{Im}(B^H_{\ell 0}(t))=0$ for all $\ell\in\mathbb N$ and $t\geq 0$. It has the following covariance structure  (see e.g. \cite{ABOW} for more details):
\[\mathbb E \left[B_t^H(x) B_t^H(y)\right]=t^{2H}\sum_{\ell=0}^\infty A_\ell (2\ell +1)P_\ell (\cos d(x,y)).\]
For such a process, given at a certain time point $t$ (and assuming that the assumption $(A)$ holds for the sequence $A_\ell$), the limit of the estimator $\hat{c}$ would depend linearly on $t^{2H}$. Thus, one can derive a weakly consistent estimator of the Hurst parameter $H$ from observations of $B^H$ along one line on the sphere at two different time points $t$ and $s$, since
\[\frac{V_N(B^H_t)}{V_N(B^H_s)}\stackrel{N\to\infty}{\to} \left(\frac{t}{s}\right)^{2H} \text{ in probability}\]
for the respective empirical quadratic variations of $B^H_t$ and $B^H_s$.

The CLT in Theorem \ref{CLTf} (extended to joint convergence in the fractional Brownian motion example) yields asymptotic normality for such estimators by delta method.
\end{remark}

\section{Estimation of the angular power spectrum}\label{estimation}
The explicitly known expression for the expectation of $V_{N,\ell}$ gives rise to a construction of an unbiased estimator for the angular power spectrum $C_\ell$. Recall that our setting are discrete high frequency observations of $f_\ell$ \textit{along one line} on the unit sphere.
The classical error term to consider in the context of the estimation of $C_\ell$ is $\frac{\hat{C_\ell}}{C_\ell}-1$ for an estimator $\hat{C_\ell}$.

In this chapter we will analyse the asymptotic properties of the estimator defined by
\[\hat{C_\ell}:=\frac{V_{N,\ell}}{2N\frac{2\ell+1}{4\pi}\left(1-P_\ell \left(\cos \left(\frac{\pi}{2N}\right)\right)\right)}\]
 in different regimes as well as consider some other useful constructions. First, consider this estimator
in the context of Chapter \ref{fixed}, that is, for constant $\ell$. It turns out that it is not consistent in any meaningful sense. More precisely, we have
\[\frac{\hat{C_\ell}}{C_\ell}-1=\frac{V_{N,\,\ell}- \mathbb E V_{N,\,\ell}}{ \mathbb E V_{N,\,\ell}},\]
and thus,
\begin{eqnarray*}
\frac{V_{N,\,\ell}-\mathbb E V_{N,\,\ell}}{\sqrt{\Var(V_{N,\,\ell})}} = \frac{\mathbb E_{N,\,\ell}}{\sqrt{\Var(V_{N,\,\ell})}}\left(\frac{\hat{C_\ell}}{C_\ell}-1\right).
\end{eqnarray*}
 The left hand side converges weakly to a (nonzero) random variable, as shown in Theorem \ref{NCLTSphere}, and as $\mathbb E V_{N,\ell}$ and $\sqrt{\Var(V_{N,\ell})}$ are of the same order, so does $\frac{\hat{C_\ell}}{C_\ell}-1$.

Now let us move on to the growing frequency regime. There, we have the following result.

\begin{prop}
For $\ell=\ell (N)\to\infty$ the estimator $\hat{C_\ell}$ is $L^2$-consistent if $\frac{\log N}{\ell}\to 0$ (and strongly consistent if the decay is polynomial) and we have
\[d \left(c\sqrt{\frac{\ell}{\log N}}\left(\frac{\hat{C_\ell}}{C_\ell}-1)\right),N(0,1)\right)\lesssim \frac{1}{\log N},\]
where $c$ is a (known) absolute constant that can vary depending on the regime.
\end{prop}
\begin{proof}
$L^2$-consistency of this unbiased estimator follows by computing
\[\mathbb E \left[\left(\frac{\hat{C_\ell}}{C_\ell}-1\right)^2\right]= \Var \left(\frac{\hat{C_\ell}}{C_\ell}\right)\sim_c \frac{\log N}{\ell} \]
from the expectation and variance asymptotics derived in the previous chapter. Under polynomial decay of $ \frac{\log N}{\ell}$ strong consistency is a direct consequence of this fact obtained by hypercontractivity allowing a standard Borel-Cantelli argument.
The quantitative CLT follows from Theorem \ref{CLTfl} due to the relation
\begin{eqnarray*}
\frac{V_{N,\,\ell}-\mathbb E V_{N,\,\ell}}{\sqrt{\Var(V_{N,\,\ell})}} = \frac{\mathbb E V_{N,\,\ell}}{\sqrt{\Var(V_{N,\,\ell})}}\left(\frac{\hat{C_\ell}}{C_\ell}-1\right),
\end{eqnarray*}
since
\[\frac{ \mathbb E V_{N,\,\ell}}{\sqrt{\Var(V_{N,\,\ell})}}\]
is of order $\sqrt{\frac{\ell}{\log N}}$ in all regimes.
\end{proof}

\begin{remark}
As mentioned in the introduction, the problem of estimating the angular power spectrum from Fourier coefficients of $f$ has a standard solution that is well-studied and documented, for instance, in \cite{MP}. Namely, given an observation of the Fourier coefficients $(a_{\ell m})$ for $m=-\ell ,\dots ,\ell$ and $\ell\in\mathbb N$, one can estimate $C_\ell$ with $\tilde{C_\ell}:=\frac{1}{2\ell +1}\sum_{m=-\ell}^{\ell}|a_{\ell m}|^2$. The variance of this estimator decays as $\frac{1}{2\ell +1}$, and one obtains
\[d_{TV}\left(\sqrt\frac{2\ell +1}{2}\left(\frac{\tilde{C_\ell}}{C_\ell}-1\right), N(0,1)\right)\leq \sqrt\frac{8}{2\ell +1}.\]
The estimators defined above can be trivially recovered from the Fourier coefficients of $f$ and use less information (assuming that one has access to observations of $f_\ell$ over a single line). For every choice of $N\to\infty$ the variance of $\hat{C_\ell}$ is higher than that of $\tilde{C_\ell}$, however, for $N$ growing at most polynomially in $\ell$ the difference is small. Note, moreover, that there is a tradeoff between the speeds of convergence for the second and the third order limits. Thus, by choosing $\log N=\ell^x$ with $x\in (0.5,1)$ we obtain an estimator whose mean squared error converges to zero as $\ell^{x-1}$, but whose distance in total variation to $N(0,1)$ can be bounded by $\ell^{-x}$, which is a faster rate than that of the classical estimator.
\end{remark}

Since for practical reasons the computation of $P_\ell$ necessary for the implementation of $\hat{C_\ell}$ might be memory-consuming, one might be interested in considering alternative, asymptotically unbiased estimators defined using the approximate expressions for the expectation, namely
\begin{eqnarray*}
\hat{C_\ell}^{(1)}&:=&\frac{V_{N,\ell}}{2N\frac{2\ell+1}{4\pi}},\\
\hat{C_\ell}^{(2)}&:=&\frac{V_{N,\ell}}{2N\frac{2\ell+1}{4\pi}(1-J_0(\frac{\pi}{2}c))},\\
\hat{C_\ell}^{(3)}&:=&\frac{V_{N,\ell}}{2N\frac{2\ell+1}{4\pi}\frac{\pi^2}{8}\frac{\ell^2}{N^2}}
\end{eqnarray*}
in the cases $N/\ell \to 0$, $\ell /N\to c$ and $\ell /N\to 0$ respectively. We obtain the following asymptotic results.

\begin{prop}
For $i\in{1,\,2,\,3}$ the estimators $\hat{C_\ell}^{(i)}$ are $L^2$-consistent if $\frac{\log N}{\ell}\to 0$ (and strongly consistent in many particular cases specified below) and we have
\[d \left(c\sqrt{\frac{\ell}{\log N}}\left(\frac{\hat{C_\ell}^{(i)}}{C_\ell}-1-\operatorname{bias}(\hat{C_\ell}^{(i)})\right),N(0,1)\right)\lesssim \frac{1}{\log N},\]
where $c$ is a (known) absolute constant that can vary depending on the regime and
\[\operatorname{bias}(\hat{C_\ell}^{(i)})=\mathbb E \left[\frac{\hat{C_\ell^{(i)}}}{C_\ell}-1\right].\]
\end{prop}
\begin{proof}
We begin by computing the biases of the three estimators. For $N/\ell \to 0$ we have
$$\operatorname{bias}(\hat{C_\ell}^{(1)})=\mathbb E \left[\frac{\hat{C_\ell}^{(1)}}{C_\ell}-1\right] =1-P_\ell \left(\cos \left(\frac{\pi}{2N}\right)\right)-1\sim \sqrt{\frac{N}{\ell}}$$
by Proposition \ref{P-asympt}. For $\ell/N\to c$ the bias is of order $\frac{1}{\ell}$, which corresponds to the second order term in the polynomial expansion of $P_\ell(\cos(\rho/ \ell))$ given e.g. in \cite{BT}. Finally, in the case $\ell /N\to 0$ the bias is calculated by evaluating the third order term in the Taylor expansion for $P_\ell$:
$$\operatorname{bias}(\hat{C_\ell}^{(3)})=\mathbb E \left[\frac{\hat{C_\ell}^{(3)}}{C_\ell}-1\right] =\frac{1-P_\ell \left(\cos \left(\frac{\pi}{2N}\right)\right)-\frac{\pi^2}{8}\frac{\ell^2}{N^2}}{\frac{\pi^2}{8}\frac{\ell^2}{N^2}}\sim \frac{\ell^2}{N^2}.$$
$L^2$-consistency of the three estimators now follows by direct calculation:
\[\mathbb E \left[\left(\frac{\hat{C_\ell}^{(i)}}{C_\ell}-1\right)^2\right]=\mathbb E \left[\frac{\hat{C_\ell}^{(i)}}{C_\ell}-1\right]^2 + \Var \left(\frac{\hat{C_\ell}^{(i)}}{C_\ell}\right) \lesssim \max\left(\mathbb E \left[\frac{\hat{C_\ell}^{(i)}}{C_\ell}-1\right]^2, \frac{\log N}{\ell}\right), \]
and under polynomial decay of the right hand side strong consistency follows as in the previous proposition. Note that for $\hat{C_\ell}^{(1)}$ the bias term $\frac{N}{\ell}$ always dominates, for $\hat{C_\ell}^{(2)}$ the variance $\frac{\log N}{\ell}$ takes overhand and for $\hat{C_\ell}^{(3)}$ the variance term dominates for $N$ growing faster than $\ell^{\frac{5}{4}+\varepsilon}$.

 For the quantitative CLT we note that for
$$ \frac{V_{N,\ell}C_\ell}{\hat{C_\ell}^{(i)}}=: E_{N,\,\ell}^{(i)}\sim \mathbb E V_{N,\ell}$$
we have
\[ \frac{ E_{N,\,\ell}}{\sqrt{\Var(V_{N,\,\ell})}}\left(\frac{\hat{C_\ell}^{(i)}}{C_\ell}-1-\operatorname{bias}(\hat{C_\ell}^{(i)})\right)=\frac{V_{N,\,\ell}-\mathbb E V_{N,\,\ell}}{\sqrt{\Var(V_{N,\,\ell})}},\]
and the statement follows as in the unbiased case.
Note further that $\operatorname{bias}(\hat{C_\ell}^{(i)}) \frac{ E_{N,\,\ell}}{\sqrt{\Var(V_{N,\,\ell})}}$ tends to zero in the second scenario and also in the third as soon as one has $\frac{\ell^{2.5}}{N^2\log N}\to 0$.
\end{proof}

{\bf{Acknowledgements.}}\\
The financial support of the DFG-SFB 823 is gratefully acknowledged.

\bibliography{biblio}{}
\bibliographystyle{apalike}  

\end{document}